\UseRawInputEncoding
\documentclass[12pt, reqno]{amsart}
\usepackage[margin=1in]{geometry}
\usepackage{amssymb,latexsym,amsmath,amscd,amsfonts}
\usepackage{latexsym}
\usepackage[mathscr]{eucal}
\usepackage{bm}
\usepackage{mathptmx}

\def\textmatrix#1&#2\\#3&#4\\{\bigl({#1 \atop #3}\ {#2 \atop #4}\bigr)}
\def\dispmatrix#1&#2\\#3&#4\\{\left({#1 \atop #3}\ {#2 \atop #4}\right)}
\newcommand{\beg}{\begin{equation}}
	\newcommand{\eeg}{\end{equation}}
\newcommand{\ben}{\begin{eqnarray*}}
	\newcommand{\een}{\end{eqnarray*}}

\newtheorem{thm}{Theorem}[section]
\newtheorem{cor}[thm]{Corollary}
\newtheorem{lem}[thm]{Lemma}

\newtheorem{prop}[thm]{Proposition}
\numberwithin{equation}{section} 
\theoremstyle{definition}
\newtheorem{defn}[thm]{Definition}
\newtheorem{rem}[thm]{Remark}

\newcommand{\HS}{\mathcal H}
\newcommand{\C}{\mathbb{C}}

\newcommand{\D}{\mathbb{D}}
\newcommand{\T}{\mathbb{T}}

\newcommand{\ov}{\overline}

\begin{document}
	\title[Distinguished varieties in the polydisc and dilation]
	{Distinguished varieties in the polydisc and dilation of commuting contractions}
	
	\author[Sourav Pal]{Sourav Pal}
	
	\address[Sourav Pal]{Mathematics Department, Indian Institute of Technology Bombay,
		Powai, Mumbai - 400076, India.} \email{sourav@math.iitb.ac.in , souravmaths@gmail.com}
		

	
	
	\thanks{The author was supported by a Seed Grant of IIT Bombay, the CPDA and the MATRICS Award (Award No. MTR/2019/001010) of Science and Engineering Research Board (SERB), India.}

	\begin{abstract}
		
A distinguished variety in the polydisc $\mathbb D^n$ is an affine complex algebraic variety that intersects $\mathbb D^n$ and exits the domain through the $n$-torus $\mathbb T^n$ without intersecting any other part of the topological boundary of $\mathbb D^n$. We find two different characterizations for a distinguished variety in the polydisc $\mathbb D^n$ in terms of the Taylor joint spectrum of certain linear matrix-pencils and thus generalize the seminal work due to Agler and M\raise.45ex\hbox{c}Carthy [Acta Math., 2005] on distinguished varieties in $\mathbb D^2$. We show that a distinguished variety in $\mathbb D^n$ is a part of an affine algebraic curve which is a set-theoretic complete intersection. We also show that if $(T_1, \dots , T_n)$ is commuting tuple of Hilbert space contractions such that the defect space of $T=\prod_{i=1}^n T_i$ is finite dimensional, then $(T_1, \dots , T_n)$ admits a commuting unitary dilation $(U_1, \dots , U_n)$ with $U=\prod_{i=1}^n U_i$ being the minimal unitary dilation of $T$ if and only if some certain matrices associated with $(T_1, \dots , T_n)$ define a distinguished variety in $\mathbb D^n$.
 
	\end{abstract}
	
	\maketitle
	
	\section{Introduction} \label{sec:01}

\vspace{0.2cm}

\noindent Throughout the paper, every operator is a bounded linear operator acting on a complex Hilbert space. A contraction is an operator with norm not greater one. We denote by $\D$ and $\T$, the open unit disc and unit circle of the complex plane $\C$ with center at the origin. We define spectral set, complete spectral set, rational dilation etc. in Section 2.

 In 2005, Agler and M\raise.45ex\hbox{c}Carthy published a very
influential article \cite{AM05}, where they established the fact
that for any pair of commuting contractive matrices $(T_1,T_2)$,
there is a one-dimensional complex algebraic variety $V$ in $\C^2$ such that $V$ intersects $\D^2$ and that the von-Neumann
inequality holds on $V\cap \D^2$ for any polynomial $p$ in $\C[z_1,z_2]$, i.e.
\begin{equation}\label{eqn:Intro-01}
\| p(T_1,T_2) \|\leq \sup_{(z_1,z_2)\in V \cap \D^2} |p(z_1,z_2)|,
\end{equation}
provided that $T_1,T_2$ do not have eigenvalues of unit
modulus. The variety $V$ exits the bidisc through
its distinguished boundary, the $2$-torus $\mathbb T^2$, without
intersecting any other part of its topological boundary. Such a set $V\cap \D^2$ is called a
distinguished variety in $\mathbb D^2$ and has the following determinantal representation
\begin{equation}\label{eqn:Intro-02}
V\cap \D^2=\{ (z,w)\in \mathbb D^2\,:\, \det (\Psi(z)-wI)=0 \}
\end{equation}
for some matrix-valued rational function $\Psi$ on the unit disc
$\mathbb D$ that is unitary on the unit circle $\mathbb T$. In \cite{ando}, Ando proved that for any pair of commuting contractions $T_1,T_2$ and for any polynomial $p\in\mathbb C[z_1,z_2]$,
\begin{equation}\label{eqn:ando}
\| p(T_1,T_2) \|\leq \sup_{(z_1,z_2)\in\overline{\mathbb D^2}}
|p(z_1,z_2)|.
\end{equation}
So in particular when $T_1,T_2$ are commuting matrices, Ando's inequality (\ref{eqn:ando}), which holds on the two-dimensional bidisc, also holds on the one-dimensional set $V\cap \mathbb D^2$. This is to say that the result of Agler and M\raise.45ex\hbox{c}Carthy makes a refinement of Ando's inequality for commuting contractive matrices $T_1,T_2$ without any unimodular eigenvalues. In geometric language it states that a pair of commuting matrices $(T_1,T_2)$ for which $\ov{\D^2}$ is a spectral set, can also have a curve lying in $\ov{\mathbb D^2}$ as its spectral set. This reduces a spectral set for $(T_1,T_2)$ not just by a subset but by a dimension. For a pair of commuting contractions $(T_1,T_2)$, we have the following path-breaking result due to Ando whose impact has always been extraordinary.

\begin{thm} [Ando, \cite{ando}] \label{thm:Char-Contraction}
Let $(T_1,T_2)$ be a pair of commuting operators acting on a complex Hilbert space $\HS$. Then the following are equivalent.
\begin{itemize}
\item[(i)] $(T_1,T_2)$ is a pair of contractions.

\item[(ii)] The inequality $(\ref{eqn:ando})$ holds for every scalar or matricial polynomial $p$.

\item[(iii)] There is a Hilbert space $\mathcal K$ containing $\HS$ as a closed linear subspace and a pair of commuting unitaries $(U_1,U_2)$ such that $p(T_1,T_2)=P_{\HS}\, p(U_1,U_2)|_{\HS}$ for every polynomial $p(z_1,z_2)$, where $P_{\HS}: \mathcal K \rightarrow \HS$ is the orthogonal projection. 
\end{itemize}

\end{thm}
The part-(iii) of the above theorem tells us that a pair of commuting contractions always admit a commuting unitary dilation. However, we do not have an analogue of part-(ii) and part-(iii) of Theorem \ref{thm:Char-Contraction} for a tuple of commuting contractions $(T_1, \dots , T_n)$ when $n\geq 3$ (see e.g. \cite{Par, paulsen}). Also, note that the inequality (\ref{eqn:Intro-01}) holds even for any matricial polynomial $p$ as was shown in \cite{AM05} and hence the closure of the one-dimensional variety $V\cap \D^2$ becomes a complete spectral set for such $(T_1,T_2)$.  Taking cue from such mysterious behaviour of a tuple of commuting contractions and its relation with the geometry of the polydisc $\D^n$, we make a new attempt towards generalizing the novel idea of Agler and McCarthy from two to several dimensions. To this end we mention that recently the results due to Agler and McCarthy for the bidisc have been reformulated and extended for a wider class of operators in \cite{B-K-S} and \cite{D-S} respectively.\\ 

Recall that for any positive integer $n$, $\mathbb A_{\mathbb C}^n$ denotes the \textit{affine} $n$-\textit{space} $\mathbb C^n$ over $\mathbb C$. If $S$ is any subset of the polynomial ring $\mathbb C[z_1,\dots,z_n]$, the \textit{zero set} of $S$ is defined to be the set of common zeros of all elements of $S$, namely
\[
Z(S)=\{P\in\mathbb A_{\mathbb C}^n\,:\, f(P)=0 \text{ for all } f\in S   \}.
\]
We shall follow standard terminologies from the literature of basic algebraic geometry (e.g. \cite{Ro:H}) and define the following.
\begin{defn}
A subset $W$ of $\mathbb A_{\mathbb C}^n$ is called an \textit{affine algebraic set} or simply an \textit{algebraic set} if there is a subset $T$ of $\mathbb C[z_1,\dots,z_n]$ such that $W=Z(T)$. An \textit{affine algebraic variety} or simply an \textit{algebraic variety} in $\mathbb A_{\mathbb C}^n$ is an irreducible algebraic set in $\mathbb A_{\mathbb C}^n$.
\end{defn}

\begin{defn}
A \textit{distinguished set} $\Omega$ in a domain $G\subset \mathbb C^n$ is the intersection of $G$ with an algebraic set $W\subset \mathbb A^n_{\mathbb C}$ such that the complex-dimension of $W$ is greater than zero and that $W$ exits the domain $G$ through the distinguished boundary $b\overline{G}$ of $G$ without intersecting any other part of the topological boundary $\partial \overline{G}$ of $G$, that is,
\[
\dim W >0, \quad \Omega = G \cap W \quad \text{ and } \quad \overline{\Omega}\cap \partial \overline{G}= \overline{\Omega} \cap b\overline{G},
\]
A \textit{distinguished variety} in $G$ is an irreducible distinguished set. Also for a distinguished set $\Omega$, that is defined by an algebraic set $W$, we denote by $\overline{\Omega}$ and $\partial \Omega$ the sets $W\cap \overline{G}$ and $W\cap b \ov{G} \,\;(=W\cap \partial \ov{G})$ respectively.
\end{defn}
Thus, according to our definitions the distinguished sets in $\D^2$ were studied in \cite{AM05}, not only the distinguished varieties. Nevertheless, a distinguished set is nothing but a union of distinguished varieties. Only the distinguished sets or varieties with dimension greater than zero are considered here in order to avoid triviality as the same was done in \cite{AM05}.\\

In Theorem \ref{thm:poly-01} and Corollary \ref{cor:001}, we show that a distinguished variety or a distinguished set in the polydisc $\D^n$ has complex-dimension $1$ for any $n \geq 2$. Then in Theorems \ref{thm:DVpoly} \& \ref{thm:poly-DVchar-2} we find two different representations for a distinguished variety and a general distinguished set in $\D^n$. Both these representations are obtained in terms of the Taylor joint spectrum of certain linear matrix-pencils. We also show in Corollary \ref{cor:poly-01} that a distinguished variety in the polydisc is a set-theoretic complete intersection, i.e. can have a set of $n-1$ irreducible generating polynomials. This is not true for the second representation of a distinguished set in $\D^n$ as in Theorem \ref{thm:poly-DVchar-2}. In fact, here the representation is given in terms of $n$ linear matrix-pencils that are framed with a help of $n$ projections and $n$ commuting unitaries acting on a finite dimensional Hilbert space.

We have already mentioned that parts-(ii) \& (iii) of Theorem \ref{thm:Char-Contraction} do not hold in general. Thus, one naturally asks if we have an analogue of Ando's inequality (\ref{eqn:Intro-02}) and a success of normal-boundary dilation (i.e. commuting unitary dilation) on a distinguished set in $\D^n$ when $n\geq 3$. In Theorems \ref{thm:VN}, \ref{thm:dilation-variety1} \& \ref{thm:dilation-variety2}, we show that both the questions have affirmative answers if and only if there are certain finite-dimensional projections and commuting unitaries that define a distinguished set in $\D^n$.

\vspace{0.2cm}

\section{Some basic concepts and background materials} \label{sec:prelim}

\vspace{0.3cm}

\noindent In this Section, we recall a few basic concepts associated with commuting tuple of Hilbert space operators having a certain compact subset of a complex-Euclidean space as a spectral set. We begin with the Taylor joint spectrum of a commuting operator tuple.

\subsection{The Taylor joint spectrum}

Let $\Lambda$ be the exterior algebra on $n$ generators
$e_1,...e_n$, with identity $e_0\equiv 1$. $\Lambda$ is the
algebra of forms in $e_1,...e_n$ with complex coefficients,
subject to the collapsing property $e_ie_j+e_je_i=0$ ($1\leq i,j
\leq n$). Let $E_i: \Lambda \rightarrow \Lambda$ denote the
creation operator, given by $E_i \xi = e_i \xi $ ($\xi \in
\Lambda, 1 \leq i \leq n$).
 If we declare $ \{ e_{i_1}... e_{i_k} : 1 \leq i_1 < ... < i_k \leq n \}$ to be an
 orthonormal basis, the exterior algebra $\Lambda$ becomes a Hilbert space,
 admitting an orthogonal decomposition $\Lambda = \oplus_{k=1} ^n \Lambda^k$
 where $\dim \Lambda ^k = {n \choose k}$. Thus, each $\xi \in \Lambda$ admits
 a unique orthogonal decomposition
 $ \xi = e_i \xi' + \xi''$, where $\xi'$ and $\xi ''$ have no $e_i$
contribution. It then follows that that  $ E_i ^{*} \xi = \xi' $,
and we have that each $E_i$ is a partial isometry, satisfying
$E_i^*E_j+E_jE_i^*=\delta_{i,j}$. Let $\mathcal X$ be a normed
space, let $\underline{T}=(T_1,...,T_n)$ be a commuting $n$-tuple
of bounded operators on $\mathcal X$ and set $\Lambda(\mathcal
X)=\mathcal X\otimes_{\mathbb{C}} \Lambda$. We define
$D_{\underline T}: \Lambda (\mathcal X) \rightarrow \Lambda
(\mathcal X)$ by

\[
D_{\underline T} = \sum_{i=1}^n T_i \otimes E_i .
\]
Then it is easy to see $D_{\underline T}^2=0$, so $Ran
D_{\underline T} \subset Ker D_{\underline T}$. The commuting
$n$-tuple is said to be \textit{non-singular} on $\mathcal X$ if
$Ran D_{\underline T}=Ker D_{\underline T}$.
\begin{defn}
The Taylor joint spectrum of ${\underline T}$ on $\mathcal X$ is
the set
\[
\sigma_T({\underline T},\mathcal X) = \{
\lambda=(\lambda_1,...,\lambda_n)\in \mathbb{C}^n : {\underline
T}-\lambda \text{ is singular} \}.
\]
\end{defn}
\begin{rem}
The decomposition $\Lambda=\oplus_{k=1}^n \Lambda^k$ gives rise to
a cochain complex $K({\underline T},\mathcal X)$, known as the
Koszul complex associated to ${\underline T}$ on $\mathcal X$, as
follows:
\[
K({\underline T},\mathcal X):0 \rightarrow \Lambda^0(\mathcal
X)\xrightarrow{D_{\underline T}^0}... \xrightarrow{D_{\underline
T}^{n-1}} \Lambda^n(\mathcal X) \rightarrow 0 ,
\]
where $D_{\underline T}^{k}$ denotes the restriction of
$D_{\underline T}$ to the subspace $\Lambda^k(\mathcal X)$. Thus,
\[
\sigma_T({\underline T},\mathcal X) = \{ \lambda\in \mathbb{C}^n :
K({\underline T}-\lambda ,\mathcal X)\text{ is not exact} \}.
\]
\end{rem}
For a further reading on Taylor joint spectrum an interested reader is referred to Taylor's original works, \cite{Taylor, Taylor1} or Curto's well-written survey article \cite{curto}.

\begin{rem}
Let $A_1, \dots , A_n$ be commuting complex square matrices of same order. Then the Taylor joint spectrum $\sigma_T(A_1, \dots , A_n)$ consists of joint eigenvalues of $A_1, \dots , A_n$.
\end{rem}

\subsection{Spectral and complete spectral set}

We shall follow Arveson's terminologies as in \cite{Arveson-II}. Let $X$ be a compact subset of $\mathbb
C^n$ and let $\mathcal R(X)$ denote the algebra of all rational
functions on $X$, that is, all quotients $p/q$ of polynomials $p,q$ from $\C[z_1, \dots , z_n]$ for which $q$ has no zeros in $X$. The norm of an element
$f$ in $\mathcal R(X)$ is defined as
$$\|f\|_{\infty, X}=\sup \{|f(\xi)|\;:\; \xi \in X  \}. $$
Also for each $k\geq 1$, let $\mathcal R_k(X)$ denote the algebra
of all $k \times k$ matrices over $\mathcal R(X)$. Obviously each
element in $\mathcal R_k(X)$ is a $k\times k$ matrix of rational
functions $F=[f_{ij}]$ and we can define a norm on $\mathcal
R_k(X)$ in the canonical way
$$ \|F\|=\sup \{ \|F(\xi)\|\;:\; \xi\in X \}, $$ thereby making
$\mathcal R_k(X)$ into a non-commutative normed algebra. Let
$\underline{T}=(T_1,\cdots,T_n)$ be an $n$-tuple of commuting
operators on a Hilbert space $\mathcal H$. The set $X$ is said to
be a \textit{spectral set} for $\underline T$ if the Taylor joint
spectrum $\sigma_T (\underline T)$ of $\underline T$ is a subset
of $X$ and
\begin{equation}\label{defn1}
\|f(\underline T)\|\leq \|f\|_{\infty, X}\,, \textup{ for every }
f\in \mathcal R(X).
\end{equation}
Here $f(\underline T)$ can be interpreted as $p(\underline
T)q(\underline T)^{-1}$ when $f=p/q$. Moreover, $X$ is said to be
a \textit{complete spectral set} if $\|F(\underline T)\|\leq
\|F\|$ for every $F$ in $\mathcal R_k(X)$, where $k$ is any natural number. It obvious that if $X$ is a complete spectral set for $\underline
T$ then $X$ is a spectral set for $\underline T$. 

\subsection{The distinguished boundary and rational dilation}

Let $\mathcal A(X)$ be an algebra of continuous complex-valued
functions on $X$ which separates the points of $X$. A
\textit{boundary} for $\mathcal A(X)$ is a closed subset $\Delta$ of $X$ such that every function in $\mathcal A(X)$ attains its
maximum modulus on $\Delta$. It follows from the theory of
uniform algebras that the intersection of all the boundaries of
$X$ is also a boundary for $\mathcal A(X)$ (see Theorem 9.1 of
\cite{wermer}). This smallest boundary is called the
$\check{\textup{S}}$\textit{ilov boundary} for $\mathcal A(X)$.
When $\mathcal A(X)$ is the algebra of rational functions which
are continuous on $X$, the $\check{\textup{S}}$\textit{ilov
boundary} for $\mathcal A(X)$ is called the \textit{distinguished
boundary} of $X$ and is denoted by $bX$.

A commuting $n$-tuple of operators $\underline T$ on a Hilbert
space $\mathcal H$, having $X$ as a spectral set, is said to have
a \textit{rational dilation} or \textit{normal}
$bX$-\textit{dilation} if there exists a Hilbert space $\mathcal
K$, an isometry $V:\mathcal H \rightarrow \mathcal K$ and an
$n$-tuple of commuting normal operators $\underline
N=(N_1,\cdots,N_n)$ on $\mathcal K$ with $\sigma_T(\underline
N)\subseteq bX$ such that
\begin{equation}\label{rational-dilation}
f(\underline T)=V^*f(\underline N)V, \textup{ for every } f\in
\mathcal R(X),
\end{equation}
or, in other words $f(\underline T)=P_{\mathcal H}f(\underline N)|_{\mathcal H}$ for every $f\in \mathcal R(X)$ when $\mathcal H$ is considered as a closed linear subspace of $\mathcal K$. Moreover, the dilation is called {\em minimal} if
\[
\mathcal K=\overline{\textup{span}}\{ f(\underline N) h\,:\;
h\in\mathcal H \textup{ and } f\in \mathcal R(K) \}.
\]
A celebrated
theorem due to Arveson tells us that $X$ being a complete spectral set for $\underline T$ is equivalent to the existence of a normal
$bX$-dilation of $\underline T$.

\begin{thm}[\cite{Arveson-II}, Theorrem 1.2.2 \& Corrollary] \label{Arveson}
Let $X \subset \C^n$ be compact. A commuting $n$-tuple of Hilsert space operators $(T_1, \dots , T_n)$ admits a normal $\partial X$-dilation if and only if $X$ is a complete spectral set for $(T_1, \dots ,T_n)$.
\end{thm}

\subsection{Operator theory on the polydisc}
In this Subsection, we collect from the literature a few results associated with unitary and isometric dilation of a commuting tuple of contractions. These results will be used in sequel. In this context we mention that a contraction $T$ is called \textit{pure} or $C._0$, if ${T^*}^n \rightarrow 0$ strongly as $n \rightarrow \infty$. We first state the famous Berger-Coburn-Lebow theorem that models a pure commuting isometric tuple in terms of Topelitz operators on a certain vector-valued Hardy space.

\begin{thm}[Berger-Coburn-Lebow, \cite{Ber}] \label{BCL}
	Let $V_1, \dots , V_n$ be commuting isometries on $\HS$ such that $V=\prod_{i=1}^n V_i$ is a pure isometry. Then, there exist projections $P_1, \dots , P_n$ and unitaries $U_1, \dots , U_n$ in $\mathcal B(\mathcal D_{V^*})$ such that
\[
(V_1, \dots , V_n , V) \equiv (T_{U_1P_1^{\perp}+ zU_1P_1 }, \dots , T_{U_nP_n^{\perp}+zU_nP_n }, T_z)  \;\; \text{ on } \; \; H^2(\mathcal D_{V^*}).
\] 

\end{thm}

It is clear from the definitions of the previous Subsection that if $X=\ov{\D^n}$, then a commuting tuple $(T_1, \dots , T_n)$ has $\ov{\D^n}$ as a spectral set only if $T_1, \dots , T_n$ are commuting contractions. Also, a rational dilation for such a tuple is nothing but a (commuting) unitary dilation $(W_1, \dots , W_n)$. However, an isometric dilation for such a tuple can be defined in an analogous manner where unitaries are replaced by isometries only. Here we recall from the literature a few useful theorems associated with isometric and unitary dilations of commuting contractions. 
	
	\begin{thm} [\cite{S:Pal3}, Theorem 4.1] \label{puredil}
	Let $T_1,\ldots ,T_n$ be commuting contractions on a Hilbert space $\mathcal{H}$ such that their product $T=\prod_{i=1}^nT_i$ is a $C._0$ contraction. Then $(T_1,\ldots ,T_n)$ possesses an isometric dilation $(V_1,\ldots ,V_n)$ with $V=\prod_{i=1}^nV_i$ being a minimal isometric dilation of $T$ if and only if there are unique orthogonal projections $P_1,\ldots ,P_n$ and unique commuting unitaries $U_1,\ldots ,U_n$ in $\mathcal{B}(\mathcal{D}_{T^*})$ with $\prod_{i=1}^n U_i=I$ such that the following conditions hold for $i=1, \dots , n$ :
	\begin{enumerate}
		\item $D_{T^*}T_i^*=P_i^{\perp}U_i^*D_{T^*}+P_iU_i^*D_{T^*}T^*$,
		\item $P_i^{\perp}U_i^*P_j^{\perp}U_j^*=P_j^{\perp}U_j^*P_i^{\perp}U_i^*$ ,
		\item $U_iP_iU_jP_j=U_jP_jU_iP_i$ ,
		\item  $P_1+U_1^*P_2U_1+U_1^*U_2^*P_3U_2U_1+\ldots +U_1^*U_2^*\ldots U_{n-1}^*P_nU_{n-1}\ldots U_2U_1 =I_{\mathcal{D}_{T^*}}$. 
	\end{enumerate}
	Moreover, $(V_1^*, \dots , V_n^*)$ is a co-isometric extension of $(T_1^*, \dots , T_n^*)$ and $(V_1, \dots , V_n)$ is unitarily equivalent to the comuting Toeplitz operator tuple $(T_{U_1P_1^{\perp}+zU_1P_1}, \dots , T_{U_nP_n^{\perp}+zU_nP_n})$ acting on $H^2(\mathcal D_{T^*})$.

\end{thm}

\begin{thm}  [\cite{S:Pal4}, Theorem 4.3]\label{Uni-puredil}
	Let $T_1,\ldots ,T_n$ be commuting contractions on a Hilbert space $\mathcal{H}$ such that their product $T=\Pi_{i=1}^nT_i$ is a $C._0$ contraction. Then $(T_1,\ldots ,T_n)$ possesses a unitary dilation $(W_1,\ldots ,W_n)$ with $W=\prod_{i=1}^nW_i$ being a minimal unitary dilation of $T$ if and only if there are unique orthogonal projections $P_1,\ldots ,P_n$ and unique commuting unitaries $U_1,\ldots ,U_n$ in $\mathcal{B}(\mathcal{D}_{T^*})$ with $\prod_{i=1}^n U_i=I_{\mathcal D_{T^*}}$ such that the following hold for $i=1, \dots , n$ :
	\begin{enumerate}
		\item $D_{T^*}T_i^*=P_i^{\perp}U_i^*D_{T^*}+P_iU_i^*D_{T^*}T^*$,
		\item $P_i^{\perp}U_i^*P_j^{\perp}U_j^*=P_j^{\perp}U_j^*P_i^{\perp}U_i^*$ ,
		\item $U_iP_iU_jP_j=U_jP_jU_iP_i$ ,
		\item  $P_1+U_1^*P_2U_1+U_1^*U_2^*P_3U_2U_1+\ldots +U_1^*U_2^*\ldots U_{n-1}^*P_nU_{n-1}\ldots U_2U_1 =I_{\mathcal{D}_{T^*}}$. 
	\end{enumerate} 
	Moreover, $(W_1, \dots , W_n)$ is unitarily equivalent to $(M_{U_1P_1^{\perp}+zU_1P_1}, \dots , M_{U_nP_n^{\perp}+zU_nP_n})$ acting on $L^2(\mathcal D_{T^*})$.
\end{thm}

\begin{thm} [\cite{S:Pal4}, Theorem 3.8] \label{Unimain}
	
	Let $T_1,\ldots,T_n\in \mathcal{B}(\mathcal{H})$ be commuting contractions, $T=\Pi_{j=1}^nT_j$ and $T_i'=\Pi_{i\neq j} T_j$ for $1\leq i \leq n$. 
	\begin{itemize}
		\item[(a)] If $\widetilde{\mathcal{K}}$ is the minimal unitary dilation space of $T$, then $(T_1,\ldots ,T_n)$ possesses a unitary dilation $(W_1,\ldots,W_{n})$ on $\widetilde{\mathcal{K}}$ with $W=\prod_{i=1}^{n}W_i$ being the minimal unitary dilation of $T$ if and only if there exist unique projections $P_1,\ldots ,P_n$ and unique commuting unitaries $U_1,\ldots ,U_n$ in $\mathcal{B}(\mathcal{D}_T)$ with $\prod_{i=1}^n U_i=I$ such that the following hold for $i=1, \dots, n:$
		\begin{enumerate} 
			\item $D_TT_i=P_i^{\perp}U_i^*D_T+P_iU_i^*D_TT$ ,
			\item  $P_i^{\perp}U_i^*P_j^{\perp}U_j^*=P_j^{\perp}U_j^*P_i^{\perp}U_i^*$ ,
			\item $U_iP_iU_jP_j=U_jP_jU_iP_i$ ,
			\item $D_TU_iP_iU_i^*D_T=D_{T_i}^2$ ,
			\item  $P_1+U_1^*P_2U_1+U_1^*U_2^*P_3U_2U_1+\ldots +U_1^*U_2^*\ldots U_{n-1}^*P_nU_{n-1}\ldots U_2U_1 =I_{\mathcal{D}_{T}}$.
			
		\end{enumerate} 
		Moreover, such a unitary dilation is minimal and the Hilbert space $\widetilde{\mathcal K}$ is the smallest in the sense that if $\widetilde{\mathcal H}$ is a closed linear subspace of $\widetilde{\mathcal K}$ such that $\widetilde{\mathcal H}$ is not isomorphic to $\widetilde{\mathcal K}$, then $(T_1, \dots, T_n)$ cannot possess such a unitary dilation on $\widetilde{\mathcal H}$.
		
		\item[(b)] If $(X_1, \dots , X_n)$ on $\widetilde{\mathcal K}_1$ and $(Z_1, \dots , Z_n)$ on $\widetilde{\mathcal K}_2$ are two unitary dilations of $(T_1, \dots , T_n)$ such that $X=\prod_{i=1}^nX_i$ and $Z=\prod_{i=1}^nZ_i$ are minimal unitary dilations of $T$ on $\mathcal K_1$ and $\mathcal K_2$ respectively, then there is a unitary $\widetilde{W}:\widetilde{\mathcal K}_1 \rightarrow \widetilde{\mathcal K}_2$ such that $(X_1, \dots , X_n)=(\widetilde{W}^*Z_1\widetilde{W}, \dots , \widetilde{W}^*Z_n\widetilde{W})$.
		
	\end{itemize}
\end{thm}

\vspace{0.2cm}

\section{Distinguished sets and distinguished varieties in the polydisc}

\vspace{0.3cm}

\noindent Recall that a distinguished set $\Omega$ in a domain $G\subset \mathbb C^n$ is the intersection of $G$ with an algebraic set $W$ in the affine space $\mathbb A^n_{\mathbb C}$ such that the complex-dimension of $W$ is greater than zero and that $W$ exits the domain $G$ through the distinguished boundary $b\overline{G}$ of $G$ without intersecting any other part of the topological boundary $\partial \overline{G}$ of $G$. Also, a distinguished variety in $G$ is a distinguished set determined by an irreducible algebraic set (i.e. an algebraic variety) $W$.

In this Section, we present two different representations of a distinguished set in the polydisc $\D^n$. The first one represents a distinguished variety in $\D^n$ which is a set-theoretic complete intersection and is obtained via a representation of a distinguished variety in the symmetrized polydisc $\mathbb G_n$. The second representation is an independent characterization, where the generating set consists of $n$ polynomials and hence is not a set-theoretic complete intersection. The common factor among these two representations is that both are obtained in terms of Taylor joint spectrum of linear matrix pencils.

\subsection{The first representation} We begin with a few useful facts from basic algebraic geometry.

\begin{defn}
 An \textit{affine algebraic curve} or simply an \textit{algebraic curve} in the affine space $\mathbb A^n_{\mathbb C}$ is an algebraic set that has dimension one as a complex manifold.
\end{defn}

\begin{defn}
An affine algebraic set or variety $V$ of dimension $k \,(\leq n)$ in $\mathbb A_{\mathbb C}^n$ is called a \textit{set-theoretic complete intersection} or simply a \textit{complete intersection} if $V$ is defined by $n-k$ independent polynomials in $\mathbb C[z_1.\dots, z_n]$. Thus an algebraic curve $C$ in $\mathbb A^n_{\mathbb C}$ is a complete intersection if $C$ is defined by $n-1$ independent polynomials in $\mathbb C[z_1,\dots, z_n]$.
\end{defn}
Note that in general set-theoretic complete intersection and complete intersection are different, but here we deal only with Zariski closed sets and thus here the phrase 'complete intersection' will mean the set-theoretic complete intersection.

\begin{thm}[\cite{Ro:H}, CH-I, Corollary 1.6]
Every algebraic set in $\mathbb A_{\mathbb C}^n$ can be uniquely expressed as a union of affine algebraic varieties, no one containing another.
\end{thm}

\begin{lem} [\cite{S:Pal2}, Lemma 3.4]  \label{lem:ag1}

Let $C$ be an affine algebraic curve in $\mathbb A_{\mathbb C}^n$ such that $V_b=C\cap V(z_n - b)$ contains finitely many points for any $b\in\mathbb C$, where $V(z_n-b)$ is the variety generated by $z_n-b \in \mathbb C[z_1,\dots, z_n]$. Then there exists a positive integer $\widetilde k$ such that $\# (V_b)\leq \widetilde k$ for any $b\in \mathbb C$, where $\# (V_b)$ is the cardinality of the set $V_b$.
\end{lem}

\begin{lem} [\cite{S:Pal2},, Lemma 3.5] \label{lem:ag2}
Let $X$ be an affine algebraic variety in $\mathbb A_{\mathbb C}^n$ such that $X\cap V(z_n-b)$ is a finite set for every $b\in\mathbb C$. Then $X$ is an algebraic curve.
\end{lem}
We first prove that every distinguished variety (or distinguished set) in the polydisc $\D^n$ has complex dimension $1$ which is surprising yet true for all $n \geq 2$ and is a main result of this paper.

\begin{thm}\label{thm:poly-01}
Every distinguished variety in the polydisc $\mathbb D^n$ is a part of an affine algebraic curve in $\mathbb C^n$.
\end{thm}

\begin{proof}

Let $\Omega$ be a distinguished variety in $\mathbb D^n$. Then $\Omega=V_S\cap \mathbb D^n$ for some complex algebraic variety $V_S$ in $\mathbb C^n$, generated by a set of polynomials $S\subseteq \mathbb C[z_1,\dots ,z_n]$. By Eisenbud-Evans theorem (see \cite{Eisenbud}), there exist $n$ polynomials say $f_1,\dots,f_n$ in $\mathbb C[z_1,\dots,z_n]$ such that $V_S$ is generated by $f_1,\dots,f_n$. So, without loss of generality let us assume that
\[
\Omega=V_S\cap \mathbb D^n=\{(z_1\dots,z_n)\in\mathbb D^n\,:\,f_i(z_1,\dots,z_n)=0\,,\,1\leq i \leq n  \}.
\]
Let if possible $\Omega$ has complex dimension $u$ and $u>1$. Let $(t_1,\dots,t_n)\in\Omega$. Then $|t_i|<1$ for $i=1,\dots,n$. Let
\[
\Omega_{t_n}=\{(z_1,\dots,z_n)\in \Omega \,:\, z_n=t_n   \}
\]
and let
\[
V_{S_{t_n}}=\{(z_1,\dots,z_{n-1},t_n)\in\mathbb C^n\,:\, f_i(z_1,\dots,z_{n-1},t_n)=0\,,\, i=1,\dots,n   \} \subseteq V_S.
\]
Then $\Omega_{t_n}=V_{S_{t_n}}\cap \mathbb D^n \subset V_S\cap \mathbb D^n=\Omega $. Now if every such set $\Omega_{t_n}$ is finite whenever $(t_1, \dots t_n) \in \Omega$, then by an argument similar to that in Lemma \ref{lem:ag2}, it follows that $\Omega$ has complex dimension $1$, a contradiction as $u>1$. Thus, we choose such a point $(t_1, \dots , t_n)\in \Omega$ for which $\Omega_{t_n}$ has complex dimension at least $1$. Set
\[
V_{S^{\prime}}=\{(z_1,\dots,z_{n-1})\in\mathbb C^{n-1}\,:\, f_i(z_1,\dots,z_{n-1},t_n)=0\,,\, 1\leq i \leq n   \}.
\]
Then $V_{S^{\prime}}$ is a complex algebraic variety in $\C^{n-1}$ with complex dimension at least $u-1$ and $u-1 \geq 1$. The intersection of $V_{S^{\prime}}$ and $\mathbb D^{n-1}$ is not empty as $(t_1,\dots,t_{n-1})\in V_{S^{\prime}}\cap \mathbb D^{n-1}$. Note that $V_{S^{\prime}}$, being a complex algebraic variety, is an unbounded connected subset of $\C^{n-1}$. Therefore, $V_{S^{\prime}}$ must exit through the topological boundary $\partial \overline{\mathbb D^{n-1}}$ of $\mathbb D^{n-1}$. Let $(q_1,\dots,q_{n-1})\in V_{S^{\prime}}\cap \partial \overline{\mathbb D^{n-1}}$. Then at least one of $q_1,\dots,q_{n-1}$ has modulus equal to $1$. Let $|q_t|=1$. Clearly $(q_1,\dots, q_{n-1},t_n)\in V_{S_{t_n}} \subseteq V_S$ as $(q_1,\dots q_{n-1})\in V_{S^{\prime}}\cap \partial \overline{\mathbb D^{n-1}}$. Since $|q_t|=1$ and $|t_n|<1$, $(q_1,\dots, q_{n-1},t_n)$ belongs to $\overline{\mathbb D^n}\setminus \mathbb D^n$ but does not belong to $\mathbb T^n$. Thus, $(q_1,\dots, q_{n-1},t_n) \in V_S \cap (\partial \ov{\D^n} \setminus \mathbb T^n )$ and hence $\Omega$ does not exit through the distinguished boundary of $\mathbb D^n$, a contradiction. Hence $\Omega$ has complex dimension $1$ and consequently it is an affine algebraic curve lying in $\mathbb D^n$.

\end{proof}

\begin{cor}\label{cor:001}
Every distinguished set in $\D^n$ has complex dimension $1$.
\end{cor}

\begin{proof}

An algebraic set $W$ is nothing but the union of its irreducible components which are algebraic varieties and the dimension of $W$ is the maximum of the dimensions of its irreducible components. Hence a distinguished set in $\D^n$ has complex dimension $1$.

\end{proof}

Note that Theorem \ref{thm:poly-01} and Corollary \ref{cor:001} modify Theorem 1.4 in \cite{D:Sch} due to Scheinker which was later efficiently reformed in \cite{D:Sch1}. In \cite{pal-shalit}, the author and Shalit gave an explicit description of distinguished sets in the symmetrized bidisc $
\mathbb G_2 = \{(z_1+z_2,z_1z_2):\, |z_i|<1, \; i=1,2  \}
$ as $\mathcal Z(f)\cap \mathbb G_2$, where $f(z_1,z_2)=A^*+z_2A-z_1I=0$ for a square matrix $A$ having numerical radius not greater than one.
The author of this article further generalized the result for the symmetrized polydisc $\mathbb G_n$ in \cite{S:Pal2}, where
\[
\mathbb G_n =\left\{ \left(\sum_{1\leq i\leq n} z_i,\sum_{1\leq
i<j\leq n}z_iz_j,\dots, \prod_{i=1}^n z_i \right): \,|z_i|< 1,
i=1,\dots,n \right \}.
\]
Recall from the literature (see \cite{costara1} and the references therein) that the \textit{symmetrized} $n$-\textit{disc} or simply the symmetrized polydisc $\mathbb G_n$ is the image of the polydisc $\mathbb D^n$ under the symmetrization map $\pi_n:\mathbb C^n \rightarrow \mathbb C^n$ defined by
\[
\pi_n(z)=(s_1(z),\dots, s_{n-1}(z), p(z)), \quad z=(z_1,\dots,z_n),
\]
 where
 \[
s_i(z)= \sum_{1\leq k_1 \leq k_2 \dots \leq k_i \leq n}
z_{k_1}\dots z_{k_i} \;,\; i=1,\dots,n-1 \quad \text{ and } \quad
p(z)=\prod_{i=1}^{n}z_i\,.
 \]
In \cite{S:Pal2}, we had the following repressentation of a ditinguished variety in $\mathbb G_n$.
\begin{thm} [\cite{S:Pal2}, Theorem 3.14] \label{thm:DVchar}
Let
\begin{equation}\label{eq:W}
\Lambda = \{ (s_1,\dots,s_{n-1},p)\in \mathbb G_n \,: \; (s_1,\dots,s_{n-1}) \in \sigma_T(F_1^*+pF_{n-1}\,,\, F_2^*+pF_{n-2}\,,\,\dots\,, F_{n-1}^*+pF_1) \},
\end{equation}
where $\sigma_T$ denotes the Taylor joint spectrum and $F_1,\dots,F_{n-1}$ are complex square matrices of same order that satisfy the following conditions:
\begin{itemize}
\item[(i)] $[F_i,F_j]=0$ and $[F_i^*,F_{n-j}]=[F_j^*,F_{n-i}]$, for $1\leq i<j\leq
n-1$ ; \item[(ii)] $\sigma_T(F_1^*+zF_{n-1}, F_2^*+zF_{n-2},\dots,F_{n-1}^*+zF_1, zI)\subseteq \mathbb
G_{n}$, $\forall z\in \mathbb D$ ;
\item[(iii)] the polynomials $\{ f_1,\dots ,f_{n-1} \}$, where $f_i=\det\,(F_i^*+pF_{n-i}-s_iI)$, form a regular sequence ;
\item[(iv)] either $\Lambda=V_S\cap \mathbb G_n$ or $\Lambda$ is an irreducible component of $V_S\cap \mathbb G_n$, where $V_S$ is the complex algebraic set generated by the set of polynomials $S=\{ f_1,\dots, f_{n-1} \}$.
\end{itemize}
Then, $\Lambda$ is a distinguished variety in $\mathbb G_n$, which further is an affine algebraic curve lying in $\mathbb G_n$ and a complete intersection.\\
Conversely, every distinguished variety $\Lambda$ in $\mathbb G_n$ is an affine algebraic curve in $\mathbb G_n$, a complete intersection and has representation as in $($\ref{eq:W}$)$, where $F_1,\dots, F_{n-1}$ are complex square matrices of same order satisfying the above conditions ${(i)-(iv)}$.

\end{thm}

We are on our way to find a representation of a distinguished variety in the polydisc in terms of a distinguished variety in the symmetrized polydisc. The first step towards this is the following result that guarantees that every distinguished variety in $\mathbb G_n$ gives rise to a distinguished variety in $\D^n$ and vice-versa.

\begin{thm}\label{thm:poly-02}
Let $\Lambda \subset \mathbb G_n$. Then $\Lambda$ is a distinguished variety in $\mathbb G_n$ if and only if $\Lambda=\pi_n (\Omega)$ for some distinguished variety $\Omega$ in $\mathbb D^n$.
\end{thm}

\begin{proof}

Let $\Lambda$ be a distinguished variety in $\mathbb G_n$. Then by Theorem \ref{thm:DVchar}, $\Lambda$ is an affine algebraic curve and is a complete intersection. Let $\Lambda$ be generated by $S=\{f_1,\dots, f_{n-1}  \}\subset \mathbb C[z_1,\dots,z_{n-1}]$, that is,
\[
\Lambda=\{(s_1,\dots,s_{n-1},p)\in \mathbb G_n \,:\, f_i(s_1,\dots,s_{n-1},p)=0\,,\, 1\leq i\leq n-1  \}.
\]
Let
\[
\Omega =\{(z_1,\dots,z_n)\in\mathbb D^n \,:\, f_i \circ \pi_n (z_1,\dots,z_n)=0 \,,\, 1\leq i \leq n-1 \}.
\]
Then $\Lambda =\pi_n (\Omega)$ and thus $\Omega$ is a one-dimensional complex algebraic variety in $\mathbb D^n$. Since $\pi_n(\mathbb T^n)=b\Gamma_n$, it follows that $\Omega$ exits through the distinguished boundary $\mathbb T^n$ of $\mathbb D^n$. Hence $\Omega$ is a distinguished variety in $\mathbb D^n$.\\

Conversely, suppose $\Omega$ is a distinguished variety in $\mathbb D^n$. Then by Theorem \ref{thm:poly-01}, $\Omega$ has complex dimension $1$ and by Eisenbud-Evans theorem (see \cite{Eisenbud}), there exist $n$ polynomials $g_1,\dots, g_n$ in $\mathbb C[z_1,\dots,z_n]$ that generate $\Omega$. For $i=1,\dots,n$ let
\[
\tilde{g}_i(z_1\dots,z_n)=\prod_{(\sigma(z_1),\dots,\sigma(z_n))\in S_n}\; g_i(\sigma(z_1),\dots,\sigma(z_n))\,,
\]
where $S_n$ is the group of permutation of $\{1,\dots,n  \}$. If $V_i$ and $\widetilde{V}_i$ are the algebraic sets generated by $g_i$ and $\tilde{g}_i$ respectively, then $\pi_n(\tilde{V}_i)=\pi_n(V_i)$ for each $i$. Thus if $S=\{  g_1,\dots, g_n\}$ and $S^{\prime}= \{ \tilde{g}_1,\dots,\tilde{g}_{n} \}$, then $\Omega=V_S\cap \mathbb D^n =V_{S^{\prime}}\cap \mathbb D^n$. Again since $\tilde{g}_i$ is a symmetric polynomial, $\tilde{g}_i(z_1,\dots,z_n)=f_i \circ \pi_n(z_1,\dots,z_n)$ for some polynomial $f_i$. Let
\[
\Lambda = \{ (s_1,\dots, s_{n-1},p)\in\mathbb G_n\,:\, f_i(s_1,\dots, s_{n-1},p)=0\,,\, 1\leq i \leq 1 \}.
\]
Then $\Lambda=\pi_n(\Omega)$ and $\Lambda$ is a one-dimensional complex algebraic variety in $\mathbb G_n$. Since $\pi_n(\mathbb T_n)=b\Gamma_n$, $\Lambda$ exits through the distinguished boundary $b\Gamma$. Hence $\Lambda=\pi_n(\Omega)$ is a distinguished variety in $\mathbb G_n$ and the proof is complete.

\end{proof}

\begin{cor}\label{cor:poly-01}
Every distinguished variety in the polydisc $\mathbb D^n$ is a set-theoretic complete intersection.
\end{cor}
This is straight-forward because for any distinguished variety $\Omega$ in $\mathbb D^n$, if $\Lambda=\pi_n(\Omega)$ then $\{f_1\circ \pi_n,\dots, f_{n-1}\circ \pi_n\}$ is a set of generators for $\Omega$ while $\{f_1,\dots,f_{n-1}  \}$ is a generating set for $\Lambda$.\\

Thus, combining Theorem \ref{thm:DVchar}, Theorem \ref{thm:poly-01}, Theorem \ref{thm:poly-02} and Corollary \ref{cor:poly-01}, we obtain the following representation for a distinguished variety in the polydisc. This is a main result of this paper.

\begin{thm}\label{thm:DVpoly}
A set $\Omega \subset \mathbb D^n$ is a distinguished variety in $\mathbb D^n$ if and only if $\Omega$ is an affine algebraic curve in $\mathbb D^n$, a complete intersection and there are $n-1$ complex square matrices $F_1,\dots, F_{n-1}$ of same order such that the following conditions are satisfied:

\begin{itemize}
\item[(i)] $[F_i,F_j]=0$ and $[F_i^*,F_{n-j}]=[F_j^*,F_{n-i}]$, for $1\leq i<j\leq
n-1$ ; \item[(ii)] $\sigma_T(F_1^*+zF_{n-1}, F_2^*+zF_{n-2},\dots,F_{n-1}^*+zF_1, zI)\subseteq \mathbb
G_{n}$, $\forall z\in \mathbb D$ ;
\item[(iii)] the set of polynomials $S=\{ f_1,\dots ,f_{n-1} \}\subset \mathbb C[z_1,\dots,z_n]$, where $f_i=\det\,(F_i^*+pF_{n-i}-s_iI)$, forms a regular sequence ;
\item[(iv)] either the set
\begin{align*}
\quad \quad \quad \Lambda = & \{ (s_1,\dots,s_{n-1},p)\in \mathbb G_n \,: \nonumber
\\& \; (s_1,\dots,s_{n-1}) \in \sigma_T(F_1^*+pF_{n-1}\,,\,
F_2^*+pF_{n-2}\,,\,\dots\,, F_{n-1}^*+pF_1) \},
\end{align*}
coincides with $V_S\cap \mathbb G_n$, or $\Lambda$ is an irreducible component of $V_S\cap \mathbb G_n$ ;
\item[(v)] the set $S'=\{  g_i \in \mathbb C[z_1,\dots, z_n]: g_i=f_i\circ \pi_n \,,i=1,\dots , n-1 \}$ generates $\Omega$, that is, either $\Omega=V_{S'}\cap \mathbb D^n$, or $\Omega$ is an irreducible component of $V_{S'}\cap \mathbb D^n$.

\end{itemize}

\end{thm}

\begin{proof}

Let there be $n-1$ complex square matrices of same order such that the conditions $(i)-(iv)$ are satisfied. Then $\Lambda$ is a distinguished variety in $\mathbb G_n$, by Theorem \ref{thm:DVchar}. Evidently condition-$(v)$ implies that $\Lambda = \pi_n(\Omega)$ and thus by Theorem \ref{thm:poly-02}, $\Omega$ is a distinguished variety in $\mathbb D^n$.\\
Conversely, suppose that $\Omega$ is a distinguished variety in $\mathbb D^n$. Then by Theorem \ref{thm:poly-01} and Corollary \ref{cor:poly-01}, $\Omega$ is an affine algebraic curve in $\mathbb D^n$ and is a complete intersection. Also by Theorem \ref{thm:poly-02}, there exists a distinguished variety $\Lambda$ in $\mathbb G_n$ such that $\Lambda=\pi_n(\Omega)$. Again since $\Lambda$ is a distinguished variety in $\mathbb G_n$, by Theorem \ref{thm:DVchar}, there are $n-1$ complex square matrices $F_1,\dots, F_{n-1}$ of same order such that the conditions $(i)-(iv)$ are satisfied. Condition-$(v)$ follows from the proof of Theorem \ref{thm:poly-02}. Hence the proof is complete.

\end{proof}

\begin{rem}\label{rem:AM05}
Needless to mention that the preceding theorem is a generalization of the famous result due to Agler and M\raise.45ex\hbox{c}Carthy (\cite{AM05}, Theorem 1.12) describing the distinguished sets in the bidisc $\mathbb D^2$. Their description of a distinguished set $V$ in $\mathbb D^2$ was given in terms of a matrix-valued rational function in the following way:
\[
V=\{ (z,w)\in\mathbb D^2\,:\, \det(\psi (z)-wI)=0 \},
\]
where $\psi (z)=A+zB(I-zD)^{-1}C$ is an $m \times m$-matrix-valued rational function defined on $\mathbb D$ by the entries of a unitary $(m+n)\times (m+n)$-matrix $U$ given by
\[
U=\begin{pmatrix}
A & B \\
C & D
\end{pmatrix}\,:\, \mathbb C^m \oplus \mathbb C^n \rightarrow \mathbb C^m \oplus \mathbb C^n.
\]
For $n=2$, our representation of a distinguished variety as in Theorem \ref{thm:DVpoly} is simpler and is given by determinant of a matrix-valued polynomial in the following way:
\[
V=\{ (z,w)\in\mathbb D^2\,:\, \det (A^*+Azw -(z+w)I)=0 \},
\]
where $A$ is any square matrix such that each joint eigenvalue $(s,p)$ of $(A^*+Ap, sI)$ is in the symmetrized bidisc $\mathbb G_2$.
\end{rem}

\subsection{The second representation} In this Subsection, we shall use the same notations and terminologies
as in \cite{AM05} introduced by Agler and M$^{\textup{c}}$Carthy.
We say that a function $f$ is \textit{holomorphic} on a
distinguished variety $\Lambda$ in $\mathbb G_n$, if for every
point in $\Lambda$, there is an open ball $B$ in $\mathbb C^n$
containing the point and a holomorphic function $F$ in $n$
variables on $B$ such that $F|_{B\cap \Lambda}=f|_{B \cap
\Lambda}$. We shall denote by $A(\Lambda)$ the Banach algebra of
functions that are holomorphic on $\Lambda$ and continuous on
$\overline{\Lambda}$. This is a closed unital subalgebra of
$C(\partial \Lambda)$ that separates points. The maximal ideal
space of $A(\Lambda)$ is $\overline{\Lambda}$.

For a finite measure $\mu$ on $\Lambda$, let $H^2(\mu)$ be the
closure of polynomials in $L^2(\partial \Lambda, \mu)$. If $G$ is
an open subset of a Riemann surface $S$ and $\nu$ is a finite
measure on $\overline G$, let $\mathcal A^2(\nu)$ denote the
closure in $L^2(\partial G, \nu)$ of $A(G)$. A point $\lambda$ is
said to be a \textit{bounded point evaluation} for $H^2(\mu)$ or
$\mathcal A^2(\nu)$ if evaluation at $\lambda$, \textit{a priori}
defined on a dense set of analytic functions, extends continuously
to the whole Hilbert space $H^2(\mu)$ or $\mathcal A^2(\nu)$
respectively. If $\lambda$ is a bounded point evaluation, then the
function defined by
$$ f(\lambda)=\langle f,k_{\lambda} \rangle $$
is called the \textit{evaluation functional at} $\lambda$. Below we recall a few useful results from the literature.

\begin{lem}[\cite{AM05}, Lemma 1.1]\label{basiclem1}
Let $S$ be a compact Riemann surface. Let $G\subseteq S$ be a
domain whose boundary is a finite union of piecewise smooth Jordan
curves. Then there exists a measure $\nu$ on $\partial G$ such
that every point $\lambda$ in $G$ is a bounded point evaluation
for $\mathcal A^2(\nu)$ and such that the linear span of the
evaluation functional is dense in $\mathcal A^2(\nu)$.
\end{lem}

\begin{lem} [\cite{D:Sch}, Theorem 3.1]  \label{basiclem2}
Let $\Omega$ be a one-dimensional distinguished algebraic set in
$\mathbb D^n$. Then there exists a measure $\mu$ on $\partial
\Omega $ such that every point in $\Omega$ is a bounded point
evaluation for $H^2(\mu)$ and such that the span of the bounded
evaluation functionals is dense in $H^2(\mu)$.
\end{lem}

\begin{lem} [\cite{B-K-S}, Lemma 7.6] \label{lemeval}
Let $\Omega$ be a one-dimensional distinguished algebraic set in
$\mathbb D^n$, and let $\mu$ be the measure on $\partial \Omega$
given as in Lemma \textup{\ref{basiclem2}}. A point
$(y_1,\dots,y_n) \in \mathbb D^n$ is in $\Omega$ if and only if
$(\bar y_1, \dots, \bar y_n)$ is a joint eigenvalue of the commuting tuple
$\left( M_{z_1}^*,\dots, M_{z_{n}}^* \right)$.
\end{lem}
\begin{lem}\label{lempure}
Let $\Lambda$ be a one-dimensional distinguished algebraic set in
$\mathbb D^n$, and let $\mu$ be the measure on $\partial \Omega$
given as in Lemma \textup{\ref{basiclem2}}. The multiplication
operator tuple $(M_{z_1},\dots, M_{z_{n}})$ on $H^2(\mu)$,
defined as multiplication by the co-ordinate functions, is a pure isometric tuple, i.e. $\prod_{i=1}^n M_{z_i}$ is a pure isometry on $H^2(\mu)$.
\end{lem}

\begin{proof}
Evidently $(M_{z_1}, \dots , M_{z_n})$ on $H^2(\mu)$ is a commuting tuple of isometries and so their product $\prod_{i=1}^n M_{z_i}=M_z$ is also an isometry. By a standard computation, for every
$\overline y=(y_1,\dots, y_n) \in \Omega$, the kernel function
$k_{\bar y}$ is an eigenfunction of $M_{z}^*$ corresponding to the
eigenvalue $y=y_1.y_2 \dots y_n$. Therefore,
\[
(M_{z}^*)^jk_{\overline y}=y^jk_{\overline y}
\rightarrow 0 \; \textup{ as } j \rightarrow \infty,
\]
because $|y|<1$. Since the evaluation
functionals $k_{\overline y}$ are dense in $H^2(\mu)$, this shows
that $M_{z}$ is a pure isometry.
Consequently $(M_{z_1},\dots, M_{z_{n}})$ is a pure isometric tuple on $H^2(\mu)$.

\end{proof}

In \cite{B-K-S}, a description for a distinguished set having complex dimension $1$ is given (see Theorem 7.3 in \cite{B-K-S}) in terms of a symmetric variety in $\mathbb A_\C^n$. Also, in the previous Subsection, we proved that a distinguished set in $\D^n$ is always one-dimensional. So, \cite{B-K-S} indeed describes all distinguished sets in $\D^n$. However, the second representation of a distinguished set that we are going to present here is framed differently with an independent proof. Let us mention that $\mathcal M_d(\C)$ denotes the space of $d \times d$ complex matrices.

\begin{thm} \label{thm:poly-DVchar-2}

Let $P_1, \dots , P_n$ be orthogonal projections and $U_1, \dots , U_n$ be commuting unitaries in $\mathcal M_d(\C)$ with $\prod_{i=1}^n U_i=I_{\C^d}$ satisfying
\begin{enumerate}
		\item $P_i^{\perp}U_i^*P_j^{\perp}U_j^*=P_j^{\perp}U_j^*P_i^{\perp}U_i^*$ ,
		\item $U_iP_iU_jP_j=U_jP_jU_iP_i$ ,
		\item  $P_1+U_1^*P_2U_1+U_1^*U_2^*P_3U_2U_1+\ldots +U_1^*U_2^*\ldots U_{n-1}^*P_nU_{n-1}\ldots U_2U_1 =I_{\mathcal{D}_T}$. 
	\end{enumerate}
	Then with the notation $z=\prod_{i=1}^n z_i$, the set
	\begin{equation}\label{eq:PDv-1}
\Omega =\{(z_1, \dots , z_n)\in \D^n \,: \, (z_1, \dots , z_n) \in \sigma_T(U_1P_1^{\perp}+zU_1P_1\,, \dots \,,U_nP_n^{\perp}+zU_nP_n) \}
\end{equation}
	is a distinguished set in $\D^n$.
	
	Conversely, every distinguished set in $\D^n$ is of the form (\ref{eq:PDv-1}) for a set of orthogonal projections $P_1, \dots , P_n$ and commuting unitaries $U_1, \dots , U_n$ from $\mathcal M_d(\C)$ for some $d\in \mathbb N$ with $\prod_{i=1}^n U_i=I_{\C^d}$ satisfying the above conditions $(1)-(3)$. 
	
	\end{thm}
	
	\begin{proof}
	
Suppose there are orthogonal projections $P_1, \dots , P_n$ and commuting unitaries in $\mathcal M_d(\C)$ with $\prod_{i=1}^n U_i=I_{\C^d}$ satisfying the above conditions $(1)-(3)$. Note that conditions (1) \& (2) guarantee the commutativity of the matrix pencils $U_1P_1^{\perp}+zU_1P_1\,, \dots \,,U_nP_n^{\perp}+zU_nP_n$. Thus, the Taylor joint spectrum of the matrix pencils $\sigma_T(U_1P_1^{\perp}+zU_1P_1\,, \dots \,,U_nP_n^{\perp}+zU_nP_n)$ is nonempty. Condition-(3) implies that $\prod_{i=1}^n (U_iP_i^{\perp}+zU_iP_i) = zI$. We now show that for any $z\in \D$, if $(z_1, \dots , z_n) \in \sigma_T(U_1P_1^{\perp}+zU_1P_1\,, \dots \,,U_nP_n^{\perp}+zU_nP_n)$, where $z=\prod_{i=1}^nz_i$ , then each $z_i \in \D$. Suppose $\xi$ is a unit joint eigenvector corresponding to the joint eigenvalue $(z_1, \dots , z_n)$. Then for each $i=1, \dots , n$, we have
$
(U_iP_i^{\perp}+zU_iP_i)\xi = z_i \xi.
$
Taking inner product with $\xi$ we have that
\begin{equation} \label{eqn:401}
\langle (P_i^{\perp}+zP_i)\xi , U_i^*\xi \rangle = z_i.
\end{equation}
Since $|z|<1$, we have that $\|P_i^{\perp}+zP_i\|<1$. This is because, for any unit vector $x= x_1 \oplus x_2 \in \HS$, where $x_1 \in Ran\, P_i$ and $x_2\in Ran\, P_i^{\perp}$, we have
\[
\|(P_i^{\perp}+zP_i)x\|=\| (P_i^{\perp}+zP_i)(x_1 \oplus x_2) \|=\| P_i^{\perp}x_2 \oplus zP_ix_1 \|=\sqrt{\|P_i^{\perp}x_2\|^2+ |z|\|P_ix_1\|^2} <1.
\]
Therefore, from (\ref{eqn:401}) we have that
\[
|z_i|= \left| (P_i^{\perp}+zP_i)\xi , U_i^*\xi \rangle \right| \leq \| (P_i^{\perp}+zP_i)\xi \|\| U_i^*\xi \| \leq \| P_i^{\perp}+zP_i\| <1.
\]
Therefore, $|z_i|<1$ for each $i=1, \dots , n$. This shows that the algebraic set $Z(f_1, \dots , f_n)$ generated by the polynomials $\{ f_i= U_iP_i^{\perp}+zU_iP_i\;:\; i=1, \dots , n \}$, intersects $\D^n$ and exits through the distinguished boundary $\T^n$. Hence, $\Omega$ is a distinguished set in $\D^n$.\\
 
Conversely,	suppose $\Omega$ is a distinguished set in $\D^n$. Then by Theorem \ref{thm:poly-01} and Corollary \ref{cor:001}, we have that $\Omega$ has dimension $1$. Let $\Omega=V_S\cap \mathbb D^n$ for a set of
polynomials
$S\subseteq \mathbb C[z_1,\dots,z_n]$. Since $V_S$ is an affine algebraic variety in
$\mathbb C^n$, by a celebrated theorem due to Eisenbud and Evans (see \cite{Eisenbud}), there are $n$ polynomials $f_1, \dots , f_n \in \mathbb C[z_1,\dots , z_n]$ generating $V_S$. Therefore, without loss of
generality we can choose $S$ to be the set $\{f_1,\dots, f_n\}$ and thus have
\[
\Omega =\{ (z_1,\dots, z_n)\in \mathbb D^n \,:\,
f_i(z_1,\dots,z_n)=0\,,\, i=1,\dots,n \}.
\]
We show that the polynomials $f_1,\dots ,f_n$ do not have a non-constant common factor that intersects $\mathbb D^n$. If $f_1,\dots , f_n$ have a non-constant common factor say $g$ that intersects $\mathbb D^n$ and if $(t_1, \dots , t_n) \in Z(g)\cap \D^n$, then considering the intersection of the hyperplane $z_n=t_n$ with the set $Z(g)\cap \D^n$ and following an argument similar to that in the proof of Theorem \ref{thm:poly-01}, we see that $g$ intersects $\partial \D^n \setminus \T^n$. This contradicts the fact that $\Omega$ is a distinguished set in $\D^n$. So, $f_1,\dots , f_n$ do not have any non-constant common factor.\\

Let $f_{n+1}\in \C[z_1, \dots , z_{n+1}]$ be the polynomial $z_{n+1}=\prod_{i=1}^n z_i=z_1\dots z_n$ and let $\widetilde S =\{ f_1, \dots , f_{n+1} \}$. Then
\[
V_{\widetilde S}=\{ (z_1, \dots , z_n, \Pi_{i=1}^n z_i)\,:\, (z_1, \dots , z_n)\in V_S \}.
\]
Now $V_{\widetilde S}\cap \D^{n+1} \neq \emptyset$ as $\Omega=V_S\cap \mathbb D^n \neq \emptyset $ and $V_{\widetilde S} \cap \partial \D^{n+1}= V_{\widetilde S} \cap \T^{n+1}$, because, $V_S\cap \partial \D^n=V_S \cap \T^n$. This shows that the set
\[
V_{\widetilde S}\cap \D^{n+1}=\{ (z_1, \dots , z_{n+1})\in \C^{n+1} \,:\, (z_1, \dots , z_n) \in  \Omega \; \; \& \; \; z_{n+1}=\Pi_{i=1}^n z_i \}
\]
is a distinguished set in $\D^{n+1}$. Evidently $V_{\widetilde S}$ has complex dimension $1$ and thus, for any complex number $p$, the intersection of $V_{\widetilde S}$ with the hyperplane $z_{n+1}=p$ is either equal to whole $V_{\widetilde S}$ or consists of finitely many points. If the intersection is equal to $V_{\widetilde S}$ then $V_{\widetilde S}\cap \D^{n+1}$ consists of points having the $(n+1)$-th coordinate equal to $p$. Now if $|p|<1$ then it follows that $V_{\widetilde S}$ intersects $\D^{n+1}$ but does not exit through the $(n+1)$-torus $\T^{n+1}$, a contradiction as every point on $\T^{n+1}$ is of unit modulus. On the other hand if $|p|\geq 1$, then $V_{\widetilde S}$ does not intersect $\D^{n+1}$, a contradiction to the fact that $V_{\widetilde S} \cap \D^{n+1}$ is a distinguished set in $\D^{n+1}$. Thus, the intersection of $V_{\widetilde S}$ with the hyperplane $z_{n+1}=p$ consists of finitely many points for any $p \in \C$. By Lemma \ref{lem:ag1}, there exists a positive integer $\tilde k$ such that the number of points in the intersection of $V_{\widetilde S}$ with $z_{n+1}=p$ does not exceed $\tilde k$ for any $p \in \C$. So, for the point $p\in \C$, if there are $k$ such points, say
\begin{gather*}
\left( z_{11}(p), z_{21}(p),\dots ,z_{n\,1}(p), p \right),\\
 (z_{12}(p),z_{22}(p),\dots, z_{n\,2}(p), p),\\
 \vdots \\
  (z_{1k}(p), z_{2k}(p), \dots ,z_{n \, k}(p), p),
\end{gather*}
then the points
\begin{gather*}
(z_{11}(p), z_{21}(p),\dots ,z_{n\,1}(p)),\\
(z_{12}(p),z_{22}(p),\dots, Z_{n\,2}(p)),\\
\vdots \\
(z_{1k}(p), z_{2k}(p), \dots , z_{n \, k}(p))
\end{gather*}
lies in the intersection of the sets of zeros of the following polynomials:
\begin{align*}
&g_1=(z_1-z_{11}(p))(z_1-z_{12}(p))\dots (z_1-z_{1\,k}(p))\,, \\
& g_2=(z_2-z_{21}(p))(z_2-z_{22}(p))\dots (z_2-z_{2\,k}(p))\,,\\
& \vdots \\
& g_{n}=(z_{n}-z_{n\,1}(p))(z_{n}-z_{n\, 2}(p))\dots (z_{n}-z_{n \,k}(p))\,.
\end{align*}
Considering the first polynomial, i.e. $g_1=0$ we see that 
\[
z_1^k \in {\text{span}}\,\{ 1,z_1,\dots, z_1^{k-1} \}+{Ran}\,M_{p}.
\]
Since $k \leq \widetilde k$, it follows that
\[
z_1^{\widetilde k} \in {\text{span}}\,\{ 1,z_1,\dots, z_1^{{\widetilde k}-1} \}+{Ran}\,M_{p}.
\]
Since this holds for any $z_{n+1}=p \in \C$, we have that
\[
z_1^{\widetilde k} \in {\text{span}}\,\{ 1,z_1,\dots, z_1^{{\widetilde k}-1} \}+{Ran}\,M_{z_{n+1}}.
\]
A similar argument holds for $z_2,\dots, z_{n}$ if we consider $g_2=0, \dots , g_{n}=0$ respectively. Therefore, we have
\begin{equation}\label{eqn:d01}
z_i^{\widetilde k}\in {\text{span}}\,\{ 1,z_i,\dots, z_i^{{\widetilde k}-1} \}+{Ran}\,M_{z_{n+1}}\,,\; \text{ for } i=1,\dots, n.
\end{equation}
Thus, it follows from (\ref{eqn:d01}) that for any non-negative integers $k_1,\dots , k_{n} $,
\[
z_1^{k_1}\dots z_{n}^{k_{n}}\in \text{span}\, \{ z_1^{i_1}z_2^{i_2}\dots z_{n}^{i_{n}}\,:\,0 \leq i_1,\dots,i_{n} \leq {\widetilde k}-1 \} +{Ran} \, M_{z_{n+1}}.
\]
Since $z_{n+1}=\prod_{n=1}^n z_i$ , using the symbol $\prod_{i=1}^n z_i = z$ we have that
\[
H^2(\mu) = {\text{span}}\, \{ z_1^{i_1}z_2^{i_2}\dots z_{n}^{i_{n}}\,:\,0 \leq i_1,\dots,i_{n} \leq {\widetilde k}-1 \} +{Ran}\, M_{z}\,,
\]
where $\mu$ is the measure as in Lemma \ref{basiclem2}. Now $M_{z}M_{z}^*$ is a
projection onto $Ran\, M_{z}$ and
\begin{equation}\label{eqn:essn}
Ran\,M_{z} \supseteq \{ z f(z_1,\dots,z_n):\; f \in \C[z_1,\dots, z_n] \}.
\end{equation}
Therefore, $Ran\,(I-M_{z}M_{z}^*)$, which is equal to $\mathcal D_{M_{z}^*}$, has finite dimension, say $d$. Therefore, $\mathcal D_{M_{z}^*} \equiv \C^d$. Consider the tuple of coordinate multipliers $(M_{z_1},\dots, M_{z_n})$ on $H^2(\mu)$. By Lemma \ref{lempure},
$(M_{z_1},\dots, M_{z_n})$ is a pure isometric tuple on $H^2(\mu)$. Note that the product of $M_{z_1}, \dots , M_{z_n}$ is equal to $M_z$ which is a pure isometry. Since $\mathcal D_{M_{z}^*} \equiv \C^d$, it follows from Theorem \ref{BCL} that $(M_{z_1},\dots, ,M_{z_n}, M_z)$ can be identified with
$(T_{\varphi_1},\dots,T_{\varphi_{n}},T_{z})$ on $H^2(\C^d)$, where $\varphi_i(z)=U_iP_i^{\perp}+zU_iP_i$ for each $i=1, \dots , n$, where $P_1, \dots , P_n$ are projections and $U_1, \dots , U_n$ are unitaries in $\mathcal B(\C^d)$. It remains to show that $\prod_{i=1}^n U_i=I_{\C^d}$ and that conditions $(1)-(3)$ of the theorem hold. The facts that each $T_{\varphi_i}$ is an isometry and that $T_{\varphi_i}T_{\varphi_j}=T_{\varphi_j}T_{\varphi_i}$ imply the conditions $(1)$ and $(2)$. Also, the fact that $\prod_{i=1}^n \varphi_i(z)=zI_{\C^d}$ give $\prod_{i=1}^n U_i=I_{\C^d}$ and condition-$(3)$ of the theorem (see the proof of Lemma 2.2 in \cite{Berc:Dou:Foi} for technical details). By Lemma \ref{lemeval}, a point $(t_1,\dots ,t_{n})$ is in $\Omega$ if
 and only if $(\bar t_1, \dots ,\bar t_{n})$ is a joint eigenvalue of $\left( T_{\varphi_1}^*,\dots,T_{\varphi_{n}}^* \right)$ what happens if and only if $(\bar t_1, \dots ,\bar t_{n})$ is a joint eigenvalue of $(\varphi_1(z), \dots , \varphi_n(z))$. Therefore,
\[
\Omega =\{(z_1, \dots , z_n)\in \D^n \,: \, (z_1, \dots , z_n) \in \sigma_T(U_1P_1^{\perp}+zU_1P_1\,, \dots \,,U_nP_n^{\perp}+zU_nP_n) \}.
\]
The proof is now complete.	
	
	\end{proof}
	
It is obvious from the proof of Theorem \ref{thm:poly-DVchar-2} that the generating polynomials for a distinguished set in $\D^n$ is determined by a set of projections and commuting unitaries acting on a finite dimensional Hilbert space. For this reason let us declare such a set to be a set defining a distinguished set in $\D^n$.
\begin{defn} \label{defn:Dist-1}
For any $d\geq 2$, a set $\Sigma=\{ P_1, \dots , P_n , U_1, \dots , U_n \} \subset \mathcal M_d(\C)$ consisting of orthogonal projections $P_1, \dots , P_n$ and commuting unitaries $U_1, \dots , U_n$ is said to define a distinguished set in the polydisc $\D^n$ if
	\begin{equation*}
\Omega_\Sigma =\{(z_1, \dots , z_n)\in \D^n \,: \, (z_1, \dots , z_n) \in \sigma_T(U_1P_1^{\perp}+zU_1P_1\,, \dots \,,U_nP_n^{\perp}+zU_nP_n),\, z=\Pi_{i=1}^n z_i \}
\end{equation*}
	is a distinguished set in $\D^n$.

\end{defn}

In Theorem \ref{BCL}, we have seen the role of a few projections and unitaries in characterizing a tuple of commuting isometries whose product is a pure isometry, i.e. a unilateral shift. This was even more evident in Theorems \ref{puredil}, \ref{Uni-puredil} \& \ref{Unimain} determining dilations of a tuple of commuting contractions. In the next Section, we shall see that these dilation theorems are equivalent to the existence of a distinguished set in $\D^n$ when the concerned projections and unitaries act on a finite dimensional space.

\vspace{0.3cm}

\section{Rational dilation on a distinguished set and von Neumann's inequality} \label{sec:rat-VN}

\vspace{0.3cm}

\noindent Let $\Omega$ be a distinguished set in $\D^n$. Then $\Omega = W \cap \D^n$ for an algebraic set $W$ in the affine space $\mathbb A^n_{\C}$. In this Section, we study $\ov{\Omega}$ as a spectral set. Let us consider a commuting tuple of Hilbert space operators having $\ov{\Omega} = W \cap \ov{\D^n}$ as a spectral set. We explore the possibility of success of rational dilation on $\ov{\Omega}$ followed by discovering some interesting interplay between rational dilation on the mother domain $\D^n$ and the existence of a distinguished set $\Omega$. First we show that the closure of a distinguished set is polynomially convex.

\begin{prop}\label{prop:poly-convex}
The closure of a distinguished set in $\mathbb D^n$ is polynomially convex.
\end{prop}

\begin{proof}

Suppose $\Omega$ is a distinguished set in $\D^n$ and suppose $\Omega = W \cap \D^n$ for some algebraic set $V$ in $\mathbb A^n_{\C}$. Without loss of generality let us assume that $\{ f_1,\dots ,f_n \}$ is a set of generators for $W$. Let $y=(y_1,\dots ,y_n) \in \C^n \setminus \ov{\D^n}$. Since $\ov{\D^n}$ is polynomially convex, there is a polynomial say $f \in \C[z_1,\dots ,z_n]$ such that
\[
|f(y)|>\, \sup_{z\in \ov{\D^n}}|f(z)|=\| f \|_{\infty, \ov{\D^n}} \geq \| f \|_{\infty, \ov{\Omega}}.
\] 
Now let $z=(z_1,\dots ,z_n)\in \ov{\D^n} \setminus \ov{\Omega}$. Then, there is at least one of $f_1, \dots ,f_{n}$ say $f_t$ such that $f_t(z)\neq 0$. It follows that
\[
|f_t(z)|> \, \sup_{x\in \ov{\Omega}}|f_t(x)|=\|f_t\|_{\infty, \ov{\Omega}}=0.
\]
Hence $\ov{\Omega}$ is polynomially convex.

\end{proof}

Interestingly, the definition of spectral set (or complete spectral set) reduces to the success of von Neumann's inequality whence the underlying compact set $X \subset \C^n$ is polynomially convex. We state this result below whose proof is a routine exercise and could be found in the literature, e.g. \cite{S:Pal2}.
 
\begin{lem} [\cite{S:Pal2}, Lemma 3.12] \label{poly-convex}

If $X\subseteq \mathbb C^n$ is a polynomially convex set, then $X$
is a spectral set or a complete spectral set respectively for a commuting tuple $(T_1,\dots,T_n)$ if and
only if von Neumann's inequality holds, i.e.

\begin{equation}\label{pT}
\|f(T_1,\dots,T_n)\|\leq \| f \|_{\infty,\, X}\,
\end{equation}
for all polynomials $f$ in $\C[z_1, \dots , z_n]$ or for all matricial polynomials $f=[f_{ij}]_{d \times d}$ respectively, where $d \in \mathbb N$ and each $f_{ij}\in \C[z_1, \dots , z_n]$.

\end{lem}
	
We learn from Theorem \ref{Arveson} that the success of rational dilation on a compact set $X \subset \C^n$, i.e. the existence of normal $\partial X-$dilation is equivalent to $X$ being a compact spectral set. Evidently, if $X$ is a complete spectral set for commuting operator tuple, then it is spectral set for it. As a consequence, we have von Neumann's inequality on $X$. Thus, rational dilation on $X$ is stronger than the success of von Neumann's inequality on $X$. Since we have evidences of failure of both von Neumann's inequality and rational dilation in more than two variables (e.g. see \cite{paulsen}), we can expect success of von Neumann's inequality and rational dilation in higher dimensions for certain classes of operators only. Our first main result of this Section determines such a class of commuting operator tuples in the polydisc that admits rational dilation on a distinguished set in $\D^n$.

	\begin{thm}\label{thm:VN}
Let $\Sigma=(T_1,\dots,T_n)$ be a tuple of commuting contraction with the product $T= \prod_{i=1}^n T_i $ such that $T^*$ is a $C._0$ contraction and $\dim \mathcal D_T < \infty$. If there are orthogonal projections $P_1, \dots , P_n$ and commuting unitaries $U_1, \dots , U_n$ in $\mathcal B(\mathcal D_T)$ such that the set $\Sigma = \{ P_1, \dots , P_n, U_1, \dots , U_n \}$ defines a distinguished set in $\D^n$ and satisfies
\[
D_{T^*}T_i^*=P_i^{\perp}U_i^*D_{T^*}+P_iU_i^*D_{T^*}T^* \quad \text{ for} \quad 1 \leq i \leq n,
\]
then both $(T_1, \dots , T_n)$ and $(T_1^*, \dots , T_n^*)$ possess normal $\partial \ov{\Omega}_{\Sigma}-$dilation, where $\partial \ov{\Omega}_{\Sigma} = \ov{\Omega}_{\Sigma} \setminus {\Omega}_{\Sigma}= \T^n \cap \ov{\Omega}_{\Sigma}$. Moreover, the dilation of $(T_1^*, \dots , T_n^*)$ acts on the minimal dilation space of $T^*$ and is minimal.

\end{thm}

\begin{proof}

We have that $\mathcal D_T$ has finite
dimension, say $d$. Then $\mathcal D_T \equiv \C^d$ and the projections $P_1,\dots,P_{n}$ and the commuting unitaries $ U_1, \dots , U_n$ are all complex square matrices of order $d$. Since the set $\Sigma=\{P_1, \dots , P_{n}, U_1, \dots , U_n \}$ defines a distinguished set in $\D^n$, it follows that the set
\begin{equation*}
\Omega_\Sigma =\{(z_1, \dots , z_n)\in \D^n \,: \, (z_1, \dots , z_n) \in \sigma_T(U_1P_1^{\perp}+zU_1P_1\,, \dots \,,U_nP_n^{\perp}+zU_nP_n),\, z=\Pi_{i=1}^n z_i \}
\end{equation*}
is a distinguished set in $\D^n$. Evidently $\Omega_\Sigma= V_S \cap \D^n$, where $V_S$ is the algebraic set generated by the set of polynomials $S=\{ g_i= \det (U_iP_i^{\perp}+zU_iP_i -z_iI)\,:\, 1\leq i \leq n \}$. Needless to mention that $V_S \cap \partial \mathbb D^n=V_S \cap \mathbb T^n= \partial \Omega_\Sigma=  \ov{\Omega}_\Sigma \setminus \Omega_\Sigma$ and hence
\begin{equation} \label{eqn:m-01}
\sigma_T(U_1P_1^{\perp}+zU_1P_1\,, \dots \,,U_nP_n^{\perp}+zU_nP_n) \subseteq \partial \Omega_\Sigma \subseteq \T^n,\;\;  \text{ when } \;\; z=\prod_{i=1}^n z_i \in \T.
\end{equation}
Again, since $\Sigma$ defines the distinguished variety $\Omega_\Sigma$ in $\D^n$, it follows that $\prod_{i=1}^n U_i = I_{\mathcal D_T}$ and the operators in $\Sigma$ satisfy the conditions of Theorem \ref{thm:poly-DVchar-2}. Thus, along with the given condition that $D_{T^*}T_i^*=P_i^{\perp}U_i^*D_{T^*}+P_iU_i^*D_{T^*}T^* \quad \text{ for} \quad 1 \leq i \leq n,
$, it follows from Theorem \ref{Uni-puredil} that $(T_1^*, \dots , T_n^*)$ dilates to the commuting unitary tuple $(M_{U_1P_1^{\perp}+zU_1P_1}, \dots , M_{U_nP_n^{\perp}+zU_nP_n})$ consisting of multiplication operators acting on $L^2(\mathcal D_{T^*})$ and the dilation is minimal. Combining this with Equation-(\ref{eqn:m-01}) we have that $(M_{U_1P_1^{\perp}+zU_1P_1}, \dots , M_{U_nP_n^{\perp}+zU_nP_n})$ on $L^2(\mathcal D_{T^*})$ is a minimal normal $\partial \Omega_\Sigma-$dilation for $(T_1^*, \dots , T_n^*)$.\\

It remains to prove that $(T_1, \dots , T_n)$ also possesses a normal $\partial \Omega_\Sigma -$dilation. In view of Theorem \ref{Arveson}, it suffices if we prove that $\ov{\Omega}_{\Sigma}$ is a complete spectral set for $(T_1, \dots ,T_{n})$. Note that in the previous part of this proof we have seen that the members of $\Sigma$ satisfy the hypothesis of Theorem \ref{thm:poly-DVchar-2} and also we have $D_{T^*}T_i^*=P_i^{\perp}U_i^*D_{T^*}+P_iU_i^*D_{T^*}T^* \; \text{ for} \; 1 \leq i \leq n$. Thus, it follows from Theorem \ref{puredil} that the commuting Toeplitz operator tuple $(T_{\Phi_1}, \dots , T_{\Phi_n})$ acting on $H^2(\mathcal D_{T^*})$, where $\Phi_i(z)=U_iP_i^{\perp}+zU_iP_i$ for $1\leq i \leq n$, is an isometric dilation of $(T_1^*, \dots , T_n^*)$. Moreover, $(T_{\Phi_1}^*, \dots , T_{\Phi_n}^*)$ is a co-isometric extension of $(T_1, \dots , T_n)$. Let $f$ be a matrix-valued holomorphic polynomial in
$n$-variables, where the coefficient matrices are of order $k$ say, and
let $f_*$ be the polynomial satisfying
$f_*(B_1,\dots,B_n)=f(B_1^*,\dots,B_{n}^*)^*$ for any $n$
commuting operators $B_1,\dots,B_n$. Then
\begingroup
\allowdisplaybreaks
\begin{align*}
\|f(T_1,\dots,T_{n})\| &\leq \;\; \|f(T_{\Phi_1}^*,\dots,T_{\Phi_{n}}^*)\|_{H^2( \mathcal D_{T})\otimes \mathbb C^k} \\
&= \;\;  \|f_*(T_{\Phi_1},\dots,T_{\Phi_{n}})^*\|_{H^2( \mathcal D_{T})\otimes \mathbb C^k} \\
& = \;\; \|f_*(T_{\Phi_1},\dots,T_{\Phi_{n}})\|_{H^2(\mathcal D_{T})\otimes \mathbb C^k} \\&
\leq \;\; \|f_*(M_{\Phi_1},\dots, M_{\Phi_{n}})\|_{L^2( \mathcal D_{T})\otimes \mathbb C^k} \\
& = \;\; \max_{\theta \in [0,2\pi]} \left\|f_*
\left( \Phi_1(e^{i\theta}),\dots,\Phi_{n}(e^{i\theta}) \right) \right\|.
\end{align*}
\endgroup
Since ${\Phi_1(z)},\dots, {\Phi_{n}(z)} $ are commuting normal operators for every $z\in\mathbb T$, we have that
\[
\left\| f_*\left( \Phi_1(e^{i\theta}),\dots,\Phi_{n}(e^{i\theta}) \right) \right\| =  \max_{\theta}  \{
|f_*(\lambda_1,\dots,\lambda_{n})|: (\lambda_1,\dots,\lambda_{n}) \in
\sigma_T(\Phi(e^{i\theta}),\dots,\Phi_{n}(e^{i\theta})) \}.
\]
Let us define
\[
\Omega_{\Sigma}^*= \{ (z_1,\dots,z_{n}) \in \mathbb D^n \,:
(z_1,\dots, z_{n}) \in \sigma_T(P_1^{\perp}U_1^*+zP_1U_1^*, \dots, P_n^{\perp}U_n^*+zP_nU_n^*)
\}.
\]
Then following the proof of Theorem \ref{thm:poly-DVchar-2} one can easily find that $\Omega_\Sigma^*$ is also a distinguished set in $\D^n$. Since $(M_{\Phi_1} \dots, M_{\Phi_{n}})$ is a commuting tuple of unitaries, each joint-eigenvalue $(\lambda_1,\dots,\lambda_{n})$ in $\sigma_T(\Phi_1(e^{i\theta}),\dots,\Phi_{n}(e^{i\theta}))$ belongs to $\T^n$. Therefore, we have
\begingroup
\allowdisplaybreaks
\begin{align*}
\| f(T_1,\dots, T_{n}) \| &\leq \quad \max_{(z_1,\dots,
z_{n})\in \overline{\Omega_{\Sigma}^*}
\cap \T^n} \;\; \| f_*(z_1,\dots,z_{n}) \| \\
& = \quad \max_{(z_1,\dots,z_{n})\in \overline{\Omega_{\Sigma}^*}
\cap \T^n} \;\; \|
\overline{f(\bar{z}_1,\dots,\bar{z}_n)} \| \\
& = \quad \max_{(z_1,\dots,z_{n})\in \overline{\Omega_{\Sigma}^*} \cap \T^n} \;\;  \| f(\bar{z}_1,\dots,\bar{z}_n) \| \\
& = \quad \max_{(z_1,\dots,z_{n})\in \overline{\Omega}_{\Sigma} \cap \D^n} \;\;  \|
f(z_1,\dots,z_{n}) \| \\
&  \leq \quad \max_{(z_1,\dots,z_{n})\in
\overline{\Omega}_{\Sigma}} \;\; \quad \| f(z_1,\dots,z_{n}) \|.
\end{align*}
\endgroup

Hence, by Lemma \ref{poly-convex}, $\ov{\Omega}_{\Sigma}$ is a complete spectral set for $(T_1, \dots ,T_{n})$ and the proof is complete.

\end{proof}

We now ask if the converse of Theorem \ref{thm:VN} is true, that is, if a tuple of commuting contractions $(T_1, \dots , T_n)$ admits a normal $\partial \Omega-$dilation for a distinguished set $\Omega$ in $\D^n$, then does there exist a set of orthogonal projections and commuting unitaries $\Sigma= \{ P_1, \dots , P_n, U_1, \dots , U_n\} \subset \mathcal B(\mathcal D_T)$ such that $\Omega=\Omega_\Sigma$ ? The following results that present an interplay between dilation of a tuple of commuting contractions and a certain distinguished set determined by a set $\Sigma$, show that the answer to the above question is affirmative.

\begin{thm} \label{thm:dilation-variety1}

Let $(T_1, \dots , T_n)$ be commuting Hilbert space contractions such that the product ${ T=\prod_{i=1}^n T_i }$ is a $C._0$ contraction and that $\dim \mathcal D_{T^*} < \infty $. Then the following are equivalent.\\

\begin{enumerate}

\item $(T_1, \dots , T_n)$ possesses a unitary dilation $(W_1, \dots , W_n)$ with ${ W=\prod_{i=1}^n W_i }$ being the minimal unitary dilation of $T$.\\

\item  There are orthogonal projections $P_1,\ldots ,P_n$ and commuting unitaries $U_1,\ldots ,U_n$ in $\mathcal{B}(\mathcal{D}_{T^*})$ with $\prod_{i=1}^n U_i=I_{\mathcal D_{T^*}}$ such that $D_{T^*}T_i^*=P_i^{\perp}U_i^*D_{T^*}+P_iU_i^*D_{T^*}T^*$ for $1 \leq i \leq n$ and the set $\Sigma=\{ P_1, \dots , P_n, U_1, \dots , U_n \}$ defines a distinguished set in $\D^n$.\\

\item $(T_1, \dots, T_n)$ possesses a normal $\partial {\Omega}_\Sigma -$dilation for a distinguished set $\Omega_\Sigma$ given by
\begin{equation*}
\quad \quad \Omega_\Sigma =\{(z_1, \dots , z_n)\in \D^n \,: \, (z_1, \dots , z_n) \in \sigma_T(U_1P_1^{\perp}+zU_1P_1\,, \dots \,,U_nP_n^{\perp}+zU_nP_n),\, z=\Pi_{i=1}^n z_i \},
\end{equation*}
where $P_1, \dots , P_n$ are orthogonal projections and $U_1, \dots , U_n$ are commuting unitaries in $\mathcal B(D_{T^*})$ such that $\prod_{i=1}^n U_i=I_{\mathcal D_{T^*}}$ and that $D_{T^*}T_i^*=P_i^{\perp}U_i^*D_{T^*}+P_iU_i^*D_{T^*}T^*$ for $1 \leq i \leq n$.
\end{enumerate}

\end{thm}

\begin{proof}

\textbf{(1) $\Rightarrow$ (2)} Since $(T_1, \dots , T_n)$ dilates to a commuting tuple of unitaries $(W_1, \dots , W_n)$ with $W=\prod_{i=1}^n W_i$ being the minimal unitary dilation of $T= \prod_{i=1}^n T_i$, it follows from Theorem \ref{Uni-puredil} that there are projections $P_1, \dots , P_n$ and commuting unitaries $U_1, \dots , U_n$ in $\mathcal B(\mathcal D_{T^*})$ satisfying $\prod_{i=1}^n U_i=I_{\mathcal D_{T^*}}$ and the four conditions of Theorem \ref{Uni-puredil}. Evidently, condition $(1)$ provides $D_{T^*}T_i^*=P_i^{\perp}U_i^*D_{T^*}+P_iU_i^*D_{T^*}T^*$ for $1 \leq i \leq n$. Also, conditions $(2)-(4)$, by an application of Theorem \ref{thm:poly-DVchar-2}, imply that the set $\Sigma = \{P_1, \dots , P_n, U_1, \dots, U_n \}$ defines a distinguished set in $\mathbb D^n$.\\

\noindent \textbf{(2)$ \Rightarrow$ (3)} Evidently, the set $\Sigma = \{P_1, \dots , P_n, U_1, \dots , U_n \}$ defines the distinguished set $\Omega_\Sigma$ in $\D^n$ by Theorem \ref{thm:poly-DVchar-2}. The fact that $(T_1, \dots , T_n)$ possesses a normal $\partial {\Omega}_\Sigma-$dilation follows from Theorem \ref{thm:VN}. \\

\noindent \textbf{(3) $\Rightarrow$ (1)} Note that if $\Omega_\Sigma$ is a distinguished set in $\mathbb D^n$, it follows from Theorem \ref{thm:poly-DVchar-2} that the projections $P_1, \dots , P_n$ and commuting unitaries $U_1, \dots , U_n$ from $\mathcal B(\mathcal D_{T^*})$ satisfy conditions $(2)-(4)$ of Theorem \ref{Uni-puredil}. Thus, along with the condition $D_{T^*}T_i^*=P_i^{\perp}U_i^*D_{T^*}+P_iU_i^*D_{T^*}T^*$ for $1 \leq i \leq n$, we have $(3) \Rightarrow (1)$ by Theorem \ref{Uni-puredil}.

\end{proof}

The following result is a next step to Theorem \ref{thm:dilation-variety1} in the sense that here we remove the condition that $T$ is a $C._0$ contraction and present a more generalized interplay between dilation and distinguished sets.

\begin{thm} \label{thm:dilation-variety2}

Let $(T_1, \dots , T_n)$ be commuting Hilbert space contractions such that the product such that $\dim \mathcal D_{T} < \infty $. Then the following are equivalent.\\

\begin{enumerate}

\item $(T_1, \dots , T_n)$ possesses a unitary dilation $(W_1, \dots , W_n)$ with ${ W=\prod_{i=1}^n W_i }$ being the minimal unitary dilation of $T$.\\

\item  There are orthogonal projections $P_1,\ldots ,P_n$ and commuting unitaries $U_1,\ldots ,U_n$ in $\mathcal{B}(\mathcal{D}_{T})$ with $\prod_{i=1}^n U_i=I_{\mathcal D_{T}}$ such that $D_TT_i=P_i^{\perp}U_i^*D_T+P_iU_i^*D_TT$ and $D_TU_iP_iU_i^*D_T=D_{T_i}^2$ for $1 \leq i \leq n$ and the set $\Sigma=\{ P_1, \dots , P_n, U_1, \dots , U_n \}$ defines a distinguished set in $\D^n$.

\end{enumerate}
\end{thm}

\begin{proof}

\textbf{(1) $\Rightarrow$ (2)} Since $(T_1, \dots , T_n)$ possesses a unitary dilation $(W_1, \dots , W_n)$ with $W=\prod_{i=1}^n W_i$ being the minimal unitary dilation of $T=\prod_{i=1}^nT_i$, it follows from Theorem \ref{Unimain} that there are projections $P_1, \dots , P_n$ and commuting unitaries $U_1, \dots , U_n$ in $\mathcal B(\mathcal D_{T^*})$ such that $\prod_{i=1}^n U_i=I_{\mathcal D_T}$ and the conditions $(1)-(5)$ (of Theorem \ref{Unimain}) hold. Evidently, the conditions $(2), (3)$ and $(5)$ along with Theorem \ref{thm:poly-DVchar-2} imply that the set $\Sigma= \{ P_1, \dots , P_n, U_1, \dots , U_n \}$ defines the following distinguished set in $\D^n$:
\begin{equation*}
\quad  \Omega_\Sigma =\{(z_1, \dots , z_n)\in \D^n \,: \, (z_1, \dots , z_n) \in \sigma_T(U_1P_1^{\perp}+zU_1P_1\,, \dots \,,U_nP_n^{\perp}+zU_nP_n),\, z=\Pi_{i=1}^n z_i \}.
\end{equation*}

\noindent \textbf{(2) $\Rightarrow$ (1)} It suffices to prove that conditions $(2), (3)$ and $(5)$ of Theorem \ref{Unimain} hold. Since $\Sigma= \{ P_1, \dots , P_n, U_1, \dots , U_n \}$ defines the distinguished set $\Omega_\Sigma$ in $\D^n$, the desired conditions $(2), (3), (5)$ follow from Theorem \ref{thm:poly-DVchar-2}. Hence the proof is complete.

\end{proof}

We conclude this article here. Our future line of works in this direction will be based on the following open problems.
\begin{enumerate}
\item In \cite{AM05}, Agler and McCarthy have shown a sharpening of Ando's inequality (i.e. von Neumann's inequality) for a pair of commuting contractive matrices $(T_1, T_2)$ on a distinguished set in $\D^2$. However, it is too much to expect an analogue of this for a tuple of commuting contractive matrices $(T_1, \dots , T_n)$. The challenge is if we can characterize the class of commuting contractive matrices which admit von Neumann's inequality or normal boundary dilation on a distinguished set in $\D^n$ when $n$ is greater than $2$.\\

\item In \cite{pal-shalit} and \cite{D-S}, it was shown that rational dilation can succeed on a distinguished variety in the symmetrized bidisc $\mathbb G_2$ and the bidisc $\D^2$ respectively for a commuting pair of operators that are not necessarily matrices. It is naturally asked if we can characterize the class of commuting contractive tuples of operators that are not matrices yet have a distinguished set in $\D^n$ or $\mathbb G_n$ as a spectral or complete spectral set.

\end{enumerate} 

\vspace{0.2cm}

\noindent \textbf{Acknowledgment.} The author is thankful to Souvik Goswami and Sudhir Ghorpade for stimulating conversations on basic algebraic geometry.

\section{Data availability statement}

\begin{enumerate}
\item Data sharing is not applicable to this article as no datasets were generated or analysed
during the current study.

\item In case any datasets are generated during and/or analysed during the current study, they
must be available from the corresponding author on reasonable request.
\end{enumerate}

\vspace{0.3cm}

\noindent \textbf{Conflict of interest.} There is no conflict of interest.


\vspace{1cm}


\begin{thebibliography}{99}

\vspace{0.5cm}

\bibitem{A-K-Mc}  J. Agler, G. Knese and J. E. McCarthy, Algebraic pairs of isometries, \textit{J. Operator Theory}, 67 (2012), 215 --
236.\\

\bibitem{AM05} J. Agler and J.E. M\raise.45ex\hbox{c}Carthy,
Distinguished varieties, \textit{Acta Math.}, 194 (2005), no. 2, 133
-- 153.\\

\bibitem{AM06} J. Agler and J.E. M\raise.45ex\hbox{c}Carthy,
Parametrizing distinguished varieties, \textit{Contemp. Math.},
393 (2006), 29 -- 34.\\

\bibitem{wermer} H. Alexander and J. Wermer,
Several complex variables and Banach algebras, \textit{Graduate
Texts in Mathematics, 35; 3rd Edition}, Springer, (1997).\\

\bibitem{ando} T. Ando, On a pair of commutative contractions, \textit{Acta Sci Math} 24 (1963), 88 -- 90.\\

\bibitem{Arveson-II} W. Arveson, Subalgebras of $C^*$-algebras II,
\textit{Acta Math.}, 128 (1972), 271 -- 308.\\

\bibitem{B-K-S} T. Bhattacharyya, P. Kumar and H. Sau, Distinguished varieties through the Berger-Coburn-Lebow theorem, \textit{Anal. PDE}, 15 (2022), 477 -- 506.\\

\bibitem{nagy} H. Bercovici, C. Foias, L. Kerchy and B. Sz.-Nagy,
Harmonic analysis of operators on Hilbert space, Universitext,
\textit{Springer, New York}, 2010.\\

\bibitem{Berc:Dou:Foi} H. Bercovici, R.G. Douglas and C. Foias, On the classification of multi-isometries, \textit{Acta Sci.Math. (Szeged)}, 72 (2006), 639 -- 661.\\
	
\bibitem{Ber} C. A. Berger, L. A. Coburn and A. Lebow, Representation and index theory for $C^*$-algebras generated by commuting isometries, \textit{J. Functional Analysis}, 27 (1978), 51 -- 99.\\

\bibitem{costara1} C. Costara, \emph{On the spectral Nevanlinna-Pick problem}, {Studia Math.}, 170 (2005), 23--55. \\

\bibitem{curto} R. E. Curto,  Applications of several complex variables to multiparameter spectral theory,
Surveys of Some Recent Results in Operator Theory, Vol. II,
\textit{Pitman Res. Notes Math. Ser., Longman Sci. Tech., Harlow},
192 (1988), 25-90.\\

\bibitem{D-S} B. K. Das and J. Sarkar, Ando dilations, von Neumann inequality, and distinguished varieties, \textit{J. Funct. Anal.}, 272 (2017), 2114 -- 2131.\\

\bibitem{Eisenbud} D. Eisenbud and E. Graham Evans, Jr., Every algebraic set in $n-$space is the intersection of $n$ hypersurfaces, \textit{Inventiones Math.}, 19 (1973), 107 -- 112.\\

\bibitem{Ro:H} R. Hartshorne, Algebraic Geometry, \textit{Graduate Texts in Mathematics; Springer-Verlag New York}, ISBN 0-387-90244-9, $1997$.\\

\bibitem{nagy1}
Bela Sz.-Nagy, \textit{Sur les contractions de l'espace de
Hilbert}, Acta Sci. Math., 15 (1953), 87 -- 92.\\

\bibitem{Par} S. Parrott, Unitary dilations for commuting contractions, \textit{Pacific J. Math.}, 34 (1970), 481 -- 490.\\

\bibitem{paulsen} V. Paulsen, Completely Bounded Maps and Operator
Algebras, \textit{Cambridge University Press}, 2002.\\			

\bibitem{S:Pal2} S. Pal, Distinguished varieties in a family of domains associated with spectral interpolation and operator theory,  \textit{Annali della Scuola Normale Superiore di Pisa. Classe di Scienze}, To appear, DOI: 10.2422/2036-2145.202203-010.\\

\bibitem{S:Pal3} S. Pal and P. Sahasrabuddhe, Minimal isometric dilations and operator models for the polydisc, \textit{https://arxiv.org/abs/2204.11391}.\\

\bibitem{S:Pal4} S. Pal and P. Sahasrabuddhe, Minimal unitary dilation for commuting contractions, \textit{https://arxiv.org/abs/2205.09093}. \\

\bibitem{pal-shalit} S. Pal and O. M. Shalit, Spectral sets and
distinguished varieties in the symmetrized bidisc, \textit{J.
Funct. Anal.}, 266 (2014), 5779 -- 5800.\\

\bibitem{D:Sch} D. Scheinker, Hilbert function spaces and the Nevanlinna-Pick problem on the polydisc, \textit{J. Funct. Anal.}, 261 (2011), 2238 -- 2249.\\

\bibitem{D:Sch1} D. Scheinker, Corrigendum to "Hilbert function spaces and the Nevanlinna-Pick problem on the polydisc'' [J. Funct. Anal. 261 (8) (2011) 2238-2249], \textit{J. Funct. Anal.}, 282 (2022), no. 4, Paper No. 109330, 2 pp.\\

\bibitem{Taylor} J. L. Taylor, The analytic-functional calculus for several commuting operators, \textit{Acta Math.} 125 (1970), 1 -- 38.\\

\bibitem{Taylor1} J. L. Taylor, A joint spectrum for several commuting operators, \textit{J. Funct. Anal.}, 6 (1970), 172 -- 191.\\

\bibitem{von-Neumann}
J. von Neumann, \textit{Eine Spektraltheorie f\"{u}r allgemeine Operatoren eines unit\"{a}ren Raumes}, Math. Nachr., 4 (1951), 258 -- 281.


\end{thebibliography}
\end{document}